\documentclass{article}%
\usepackage{amsfonts}
\usepackage{amsmath}
\usepackage{amssymb}
\usepackage{graphicx}%
\newtheorem{theorem}{Theorem}

\newtheorem{proposition}[theorem]{Proposition}

\newenvironment{proof}[1][Proof]{\noindent\textbf{#1.} }{\ \rule{0.5em}{0.5em}}
\begin{document}

\title{The Sums of a Double Hypergeometric Series and of the First $m+1$ Terms of
$_{3}F_{2}\left(  a,b,c;\left(  a+b+1\right)  /2,2c;1\right)  $ when $c=-m$ is
a Negative Integer}
\author{Charles F. Dunkl\thanks{E-mail: cfd5z@virginia.edu,
URL:http://people.virginia.edu/\symbol{126}cfd5z}\\Dept. of Mathematics, University of Virginia
\and George Gasper\thanks{E-mail: george@math.northwestern.edu,
URL:http://math.northwestern.edu/\symbol{126}george}\\Dept. of Mathematics, Northwestern University}
\maketitle

\begin{abstract}
A summation formula is derived for the sum of the first $m+1$ terms of the
$_{3}F_{2}\left(  a,b,c;\left(  a+b+1\right)  /2,2c;1\right)  $ series when
$c=-m$ is a negative integer. This summation formula is used to derive a
formula for the sum of a terminating double hypergeometric series that arose
in another project by one of us (C.D.)

\end{abstract}

\section{Introduction}

In the process of proving the terminating double hypergeometric summation
formula in Proposition \ref{dblsum} below, we needed to derive a summation
formula for a special case of the sum of the first $m+1$ terms of the
$_{3}F_{2}\left(  a,b,c;\frac{a+b+1}{2},2c;1\right)  $ series when $c=-m$ is a
negative integer. This $_{3}F_{2}$ series also appears in Watson's summation
formula%
\begin{equation}
_{3}F_{2}\left(
\genfrac{}{}{0pt}{}{a,b,c}{\frac{a+b+1}{2},2c}%
;1\right)  =\frac{\Gamma\left(  \frac{1}{2}\right)  \Gamma\left(  \frac{1}%
{2}+c\right)  \Gamma\left(  \frac{1}{2}+\frac{a}{2}+\frac{b}{2}\right)
\Gamma\left(  \frac{1}{2}-\frac{a}{2}-\frac{b}{2}+c\right)  }{\Gamma\left(
\frac{1}{2}+\frac{a}{2}\right)  \Gamma\left(  \frac{1}{2}+\frac{b}{2}\right)
\Gamma\left(  \frac{1}{2}-\frac{a}{2}+c\right)  \Gamma\left(  \frac{1}%
{2}-\frac{b}{2}+c\right)  },\label{WSF}%
\end{equation}
which was published by Watson in 1924 for $a$ being a negative integer, and in
1925 by Whipple for the more general case when $\operatorname{Re}\left(
2c+1-a-b\right)  >0$ for convergence and assuming $2c,\frac{a+b+1}{2}%
\notin\mathbb{Z}_{\leq0}$ (the numbers $0,-1,-2\ldots$) so that the
denominators in the terms of the series are never zero. See Bailey's book
\cite[Sec. 3.3]{GHS}.

\subsection{Remarks\label{Rem1}}

During the writing of the first version the authors found that, as the saying
goes, the journey is more educational than the destination. The starting point
was the need to sum a certain terminating truncated $_{3}F_{2}$ series; the
next logical step was to refer to standard sources, especially the
easy-to-access on-line Digital Library of Mathematical Functions \cite{DLMF},
\texttt{http://dlmf.nist.gov}. Formula (\ref{WSF}) is to be found there
\cite[16.4.6]{DLMF} but one has to avoid being too casual with applying
formulae for hypergeometric series when there are negative integers among the
denominator parameters, notably in the formulae for polynomials orthogonal
with respect to a finite discrete measure such as the Krawtchouk and Hahn
polynomials. As worked out below, formula (\ref{WSF}) does indeed contain
subtle pitfalls, besides which, one feels that a finite summation formula
should be provable without recourse to infinite series. So the next step on
the journey is to consult a knowledgeable colleague; in this case G.G. who
quickly found a proof using an 85-year-old result of Bailey's (which can
actually be found in an on-line archive - if one knows to look for it!). So
then version 1 of this note was written. But as perhaps should be expected for
something as widely used as $_{3}F_{2}$ sums, some strongly related results
(besides those of Bailey) had been previously obtained. Tom Koornwinder
pointed out that formula (\ref{sum3}) can also be deduced from Whipple's sum
\cite[16.4.7]{DLMF} by reversing the order of summation (set $k=m-j$) in the
left hand side. Some other technical comments are provided at the end of
Section 2. The authors deem it worthwhile to describe these diverse approaches
to summation problems as an instructive example of the solution process.

\subsection{Limits of Watson's Formula}

Let $k,m=0,1,2,\ldots$. We use the transformations $\Gamma\left(  a-m\right)
=\left(  -1\right)  ^{m}\dfrac{\Gamma\left(  a\right)  }{\left(  1-a\right)
_{m}}$ and $\Gamma\left(  \frac{1}{2}+t\right)  \Gamma\left(  \frac{1}%
{2}-t-m\right)  =\dfrac{\left(  -1\right)  ^{m}\pi}{\left(  \frac{1}%
{2}+t\right)  _{m}~\cos\pi t}$ (the Pochhammer symbol is defined by $\left(
t\right)  _{0}=1,\left(  t\right)  _{m+1}=\left(  t\right)  _{m}\left(
t+m\right)  $ for $t\in\mathbb{C}$), and%
\[
\lim_{c\rightarrow-m}\frac{\left(  c\right)  _{k}}{\left(  2c\right)  _{k}%
}=\frac{\left(  -m\right)  _{k}}{\left(  -2m\right)  _{k}},k=0,1,\ldots,2m,
\]
which equals zero for $k=m+1,\ldots,2m$. It follows from (\ref{WSF}) that
under the above convergence conditions%
\begin{gather}
\frac{\cos\frac{\pi a}{2}\cos\frac{\pi b}{2}}{\cos\frac{\pi\left(  a+b\right)
}{2}}\frac{\left(  \frac{a+1}{2}\right)  _{m}\left(  \frac{b+1}{2}\right)
_{m}}{\left(  \frac{1}{2}\right)  _{m}\left(  \frac{a+b+1}{2}\right)  _{m}%
}=\lim_{c\rightarrow-m}~_{3}F_{2}\left(
\genfrac{}{}{0pt}{}{a,b,c}{\frac{a+b+1}{2},2c}%
;1\right) \nonumber\\
=\sum_{k=0}^{m}\frac{\left(  a\right)  _{k}\left(  b\right)  _{k}\left(
-m\right)  _{k}}{k!\left(  \frac{a+b+1}{2}\right)  _{k}\left(  -2m\right)
_{k}}+\lim_{c\rightarrow-m}\sum_{k=2m+1}^{\infty}\frac{\left(  a\right)
_{k}\left(  b\right)  _{k}\left(  c\right)  _{k}}{k!\left(  \frac{a+b+1}%
{2}\right)  _{k}\left(  2c\right)  _{k}}. \label{sum2}%
\end{gather}
Thus we see that deriving a summation formula for the first sum in the right
hand side of (\ref{sum2}) from formula (\ref{WSF}) is equivalent to the
problem of evaluating the limit of the infinite series on the right side as
$c\rightarrow-m$. In fact the termwise limit of the series is a multiple of
$_{3}F_{2}\left(
\genfrac{}{}{0pt}{}{a+2m+1,b+2m+1,m+1}{a+b+2m+\frac{3}{2},2m+2}%
;1\right)  $, which can be summed by (\ref{WSF}), and after some
simplification the value of the limiting sum is $\frac{\sin\frac{\pi a}{2}%
\sin\frac{\pi b}{2}}{\cos\frac{\pi\left(  a+b\right)  }{2}}\frac{\left(
\frac{a+1}{2}\right)  _{m}\left(  \frac{b+1}{2}\right)  _{m}}{\left(  \frac
{1}{2}\right)  _{m}\left(  \frac{a+b+1}{2}\right)  _{m}}$; a careful argument
using the dominated convergence theorem is needed to justify these operations.
Of course for reasonably small $m$ the sum can be evaluated by computer
algebra systems like Maple$\texttrademark$ or Mathematica\texttrademark\ with
the result%
\begin{equation}
\sum_{k=0}^{m}\frac{\left(  a\right)  _{k}\left(  b\right)  _{k}\left(
-m\right)  _{k}}{k!\left(  \frac{a+b+1}{2}\right)  _{k}\left(  -2m\right)
_{k}}=\frac{\left(  \frac{a+1}{2}\right)  _{m}\left(  \frac{b+1}{2}\right)
_{m}}{\left(  \frac{1}{2}\right)  _{m}\left(  \frac{a+b+1}{2}\right)  _{m}},
\label{sum3}%
\end{equation}
but this is not a proof. In Section 2 we give our proof of (\ref{sum3}) for
all nonnegative integer values of $m$, without using infinite series (and
having to justify using the termwise limit of the infinite series on the right
side of (2) as $c\rightarrow-m$ to derive (3)). The summation formula for the
previously mentioned double hypergeometric sum is considered in Section 3.

\section{The Single Sum}

Explicitly we need to prove (ignoring the trivial case $m=0$ where the sum
equals one):

\begin{proposition}
\label{3F2mm}The summation formula (\ref{sum3}) holds for $m=1,2,3\ldots$ when
the parameters $a,b$ satisfy $\frac{a+b+1}{2}\neq0,-1,-2,\ldots,1-m$.
\end{proposition}

\begin{proof}
We start with the transformation formula%
\begin{equation}
_{3}F_{2}\left(
\genfrac{}{}{0pt}{}{-m,2a,2b}{a+b+\frac{1}{2},2c}%
;1\right)  =~_{4}F_{3}\left(
\genfrac{}{}{0pt}{}{a,b,2c+m,-m}{a+b+\frac{1}{2},c,c+\frac{1}{2}}%
;1\right)  ,\label{sum4}%
\end{equation}
(equation (4.31) in Bailey's 1929 paper \cite{B}) which holds for
$m=1,2,3\ldots$ and $a+b+\frac{1}{2},2c\notin\mathbb{Z}_{\leq0},$ and where
both sums are over the first $m+1$ terms of the series. Since both series in
(\ref{sum4}) terminate it is permissible to take term-by-term limits as
$c\rightarrow-m$ giving:
\begin{align}
\sum_{k=0}^{m}\frac{\left(  -m\right)  _{k}\left(  2a\right)  _{k}\left(
2b\right)  _{k}}{k!\left(  a+b+\frac{1}{2}\right)  _{k}\left(  -2m\right)
_{k}} &  =\sum_{k=0}^{m}\frac{\left(  a\right)  _{k}\left(  b\right)
_{k}\left(  -m\right)  _{k}\left(  -m\right)  _{k}}{k!\left(  a+b+\frac{1}%
{2}\right)  _{k}\left(  -m\right)  _{k}\left(  -m+\frac{1}{2}\right)  _{k}%
}\label{sumPF}\\
&  =~_{3}F_{2}\left(
\genfrac{}{}{0pt}{}{a,b,-m}{a+b+\frac{1}{2},-m+\frac{1}{2}}%
;1\right)  \nonumber\\
&  =\frac{\left(  a+\frac{1}{2}\right)  _{m}\left(  b+\frac{1}{2}\right)
_{m}}{\left(  a+b+\frac{1}{2}\right)  _{m}\left(  \frac{1}{2}\right)  _{m}%
}\nonumber
\end{align}
by the Pfaff-Saalsch\"{u}tz summation formula (see \cite[2.2(1)]{GHS}).
Replacing $a,b$ in (\ref{sumPF}) by $\frac{a}{2},\frac{b}{2}$ respectively
completes the proof.
\end{proof}

The series in (\ref{sum3}) can also be written in the Bailey notation
\cite[Sec. 10.4]{GHS} as $_{3}F_{2}\left(  a,b,-m;\frac{a+b+1}{2}%
,-2m;1\right)  $ to $m+1$ terms, or as a truncated\linebreak$_{3}F_{2}\left(
a,b,-m;\frac{a+b+1}{2},-2m;1\right)  _{m}$ series, where the subscript $m$
denotes the sum of the first $m+1$ terms of the $_{3}F_{2}$ series. Before
writing the first version of this note, we had searched for the summation
formula (\ref{sum3}) in many papers and books on special functions, including
well-known books by G.E. Andrews and R. Askey and R. Roy \cite{AAR}, R.P.
Agarwal, W.N. Bailey, A. Erd\'{e}lyi (\emph{Higher Transcendental Functions}),
G. Gasper and M. Rahman, I.S. Gradshteyn and I.M. Ryzhik, M.E.H. Ismail, Y.L.
Luke, L.J. Slater, G. Szeg\"{o}, and G.N. Watson (see, e.g., the References in
\cite{AAR} and \cite{BHS}) without finding it. 

We extended the search to basic and elliptic hypergeometric series to see if
(\ref{sum3}) could be obtained as a limit case of any published summation
formulas for such series. Eventually, this led us to the observation that from
formula (3.17) in the Jain paper \cite{J1981}
\begin{equation}
_{4}\phi_{3}\left(
\genfrac{}{}{0pt}{}{a^{2},b^{2},-q^{-N},q^{-N}}{ab\sqrt{q},-ab\sqrt{q}%
,q^{-2N}}%
;q;q\right)  _{N}=\frac{\left(  a^{2}q;q^{2}\right)  _{N}\left(  b^{2}%
q;q^{2}\right)  _{N}}{\left(  a^{2}b^{2}q;q^{2}\right)  _{N}\left(
q;q^{2}\right)  _{N}},\label{sum6}%
\end{equation}
where we inserted a missing subscript $N$ on the right side of the series to
indicate that it is a truncated $_{4}\phi_{3}$ series, it follows by replacing
$a,b$ in it by $q^{a},q^{b}$ respectively, and letting $q\rightarrow1$ that
this formula is a $q$-analogue of (\ref{sum3}). Analogous to Koornwinder's
observation mentioned in Sec. \ref{Rem1}, formula (\ref{sum6}) can be deduced
from Andrews' $q$-analogue of Whipple's sum in \cite{An} by reversing the
order of summation on the left hand side (see the formula for reversing the
summation order in terminating $q$-series in \cite[Ex. 1.4(ii)]{BHS}). For a
nonterminating $q$-analogue of (1), see Ex. 2.17 in \cite{BHS}.

\section{The Double Sum}

Another project of one of us (C.D.) deals with evaluating separability
probabilities for $4\times4$ so-called X-density matrices, a continuation of
investigations by C. D. and P. Slater, \cite{SD1},\cite{SD2}. These matrices
form a 7-dimensional subset of the space $M_{4}\left(  \mathbb{C}\right)  $
and the probability calculation involves a five-fold iterated integral, which
leads to the definite integral of a double sum containing $\int_{0}^{1}%
x^{i+j}dx.$ This can be phrased as%
\begin{equation}
S\left(  m,n\right)  :=\sum_{i=0}^{m}\frac{\left(  -m\right)  _{i}\left(
n+1\right)  _{i}}{i!\left(  m+n+2\right)  _{i}}\sum_{j=0}^{n}\frac{\left(
-n\right)  _{j}\left(  \frac{1}{2}-n\right)  _{j}}{j!\left(  \frac{1}%
{2}\right)  _{j}}\frac{1}{i+j+\frac{1}{2}}. \label{defS}%
\end{equation}
This series is of hypergeometric type because $\dfrac{1}{i+j+\frac{1}{2}%
}=2\dfrac{\left(  \frac{1}{2}\right)  _{i+j}}{\left(  \frac{3}{2}\right)
_{i+j}}$. For $m=0$ the formula is easily evaluated with the Chu-Vandermonde
$_{2}F_{1}$ sum (see \cite[Sec. 1.3]{GHS}), and for $n=0$ the sum is a special
case of a terminating well-poised $_{3}F_{2}$ series formula. Using symbolic
computation to evaluate $S\left(  m,n\right)  $ for $m=0,1,2,3$ suggested that
a closed form does exist; specifically the formula%
\[
\dfrac{S\left(  m,n\right)  }{S\left(  m-1,n+1\right)  }=\dfrac{2m\left(
n+1\right)  }{\left(  2n+1\right)  \left(  n+2m+1\right)  }%
\]
was verified for a few small values of $m,n$. After further exploration (which
led to trying to fill in the gap between $i+j+\frac{1}{2}$ and $j-\frac{1}{2}%
$, the last factor in $\left(  \frac{1}{2}\right)  _{j}$), a proof was found
that needed a special case of formula (\ref{sum3}).

\begin{proposition}
\label{dblsum}For $m,n=0,1,2,3,\ldots$%
\[
S\left(  m,n\right)  =2^{2m+2n}\frac{m!\left(  m+n\right)  !\left(
m+n+1\right)  !\left(  \frac{1}{2}\right)  _{n}}{n!\left(  n+2m+1\right)
!\left(  \frac{1}{2}\right)  _{m+n+1}}.
\]

\end{proposition}

\begin{proof}
Let%
\[
A_{ni}:=\sum_{j=0}^{n}\frac{\left(  -n\right)  _{j}\left(  \frac{1}%
{2}-n\right)  _{j}}{j!\left(  \frac{1}{2}\right)  _{j}\left(  i+j+\frac{1}%
{2}\right)  },
\]
and
\[
\sum_{i=0}^{k}\binom{k}{i}\frac{\left(  -1\right)  ^{i}}{i+j+\frac{1}{2}}%
=\sum_{i=0}^{k}\frac{\left(  -k\right)  _{i}\left(  j+\frac{1}{2}\right)
_{i}}{i!\left(  j+\frac{1}{2}\right)  _{i+1}}=\frac{1}{j+\frac{1}{2}}%
\frac{\left(  1\right)  _{k}}{\left(  j+\frac{3}{2}\right)  _{k}}=\frac
{k!}{\left(  j+\frac{1}{2}\right)  _{k+1}},
\]
by the Chu-Vandermonde sum. Then%
\begin{gather*}
\sum_{i=0}^{k}\binom{k}{i}\left(  -1\right)  ^{i}A_{ni}=\sum_{j=0}^{n}%
\frac{\left(  -n\right)  _{j}\left(  \frac{1}{2}-n\right)  _{j}}{j!\left(
\frac{1}{2}\right)  _{j}}\frac{k!}{\left(  j+\frac{1}{2}\right)  _{k+1}}\\
=\frac{k!}{\left(  \frac{1}{2}\right)  _{k+1}}\sum_{j=0}^{n}\frac{\left(
-n\right)  _{j}\left(  \frac{1}{2}-n\right)  _{j}}{j!\left(  \frac{1}%
{2}+k+1\right)  _{j}}=\frac{k!}{\left(  \frac{1}{2}\right)  _{k+1}}%
\frac{\left(  k+n+1\right)  _{n}}{\left(  \frac{1}{2}+k+1\right)  _{n}}\\
=\frac{k!\left(  k+n+1\right)  _{n}}{\left(  \frac{1}{2}\right)  _{k+n+1}%
}=:B_{nk}.
\end{gather*}
The matrix $M$ with $M_{ij}=\binom{i}{j}\left(  -1\right)  ^{j}$ for
$i,j\geq0$ (and $M_{ij}=0$ for $i<j$) is its own inverse, thus (note
$B_{nk}=\sum_{i=0}^{k}M_{ki}A_{ni}$)%
\[
A_{ni}=\sum_{k=0}^{i}\binom{i}{k}\left(  -1\right)  ^{k}B_{nk}.
\]
The case $S\left(  0,n\right)  =A_{n,0}=B_{n,0}$ is trivial, so we assume
$m\geq1$ in the following. The expression for $A_{nk}$ is now used to find
that :
\begin{align*}
S\left(  m,n\right)   &  =\sum_{i=0}^{m}\frac{\left(  -m\right)  _{i}\left(
n+1\right)  _{i}}{i!\left(  m+n+2\right)  _{i}}A_{ni}=\sum_{i=0}^{m}%
\frac{\left(  -m\right)  _{i}\left(  n+1\right)  _{i}}{i!\left(  m+n+2\right)
_{i}}\sum_{k=0}^{i}\left(  -1\right)  ^{k}\frac{i!}{k!\left(  i-k\right)
!}B_{nk}\\
&  =\sum_{k=0}^{m}\left(  -1\right)  ^{k}\frac{B_{nk}\left(  -m\right)
_{k}\left(  n+1\right)  _{k}}{k!\left(  m+n+2\right)  _{k}}\sum_{j=0}%
^{m-k}\frac{\left(  k-m\right)  _{j}\left(  n+1+k\right)  _{j}}{j!\left(
m+n+2+k\right)  _{j}}\\
&  =\sum_{k=0}^{m}\left(  -1\right)  ^{k}\frac{B_{nk}\left(  -m\right)
_{k}\left(  n+1\right)  _{k}}{k!\left(  m+n+2\right)  _{k}}\frac{\left(
m+1\right)  _{m-k}}{\left(  m+n+2+k\right)  _{m-k}}\\
&  =\frac{\left(  m+1\right)  _{m}}{\left(  m+n+2\right)  _{m}}\sum_{k=0}%
^{m}\frac{k!\left(  k+n+1\right)  _{n}\left(  -m\right)  _{k}\left(
n+1\right)  _{k}}{k!\left(  \frac{1}{2}\right)  _{k+n+1}\left(  -2m\right)
_{k}}\\
&  =\frac{\left(  m+1\right)  _{m}\left(  n+1\right)  _{n}}{\left(
m+n+2\right)  _{m}\left(  \frac{1}{2}\right)  _{n+1}}\sum_{k=0}^{m}%
\frac{\left(  2n+1\right)  _{k}\left(  -m\right)  _{k}\left(  1\right)  _{k}%
}{k!\left(  n+\frac{3}{2}\right)  _{k}\left(  -2m\right)  _{k}};
\end{align*}
this used the change of summation index $i=k+j$, so that $0\leq k\leq m$ and
$0\leq j\leq m-k$, and the identity
\[
\left(  m+1\right)  _{m-k}=\dfrac{\left(  m+1\right)  _{m-k}\left(
2m+1-k\right)  _{k}}{\left(  2m+1-k\right)  _{k}}=\dfrac{\left(  m+1\right)
_{m}\left(  -1\right)  ^{k}}{\left(  -2m\right)  _{k}}.
\]
Finally we set $a=n+\frac{1}{2},b=\frac{1}{2}$ in (\ref{sumPF}) to obtain%
\begin{gather*}
\sum_{k=0}^{m}\frac{\left(  2n+1\right)  _{k}\left(  -m\right)  _{k}\left(
1\right)  _{k}}{k!\left(  n+\frac{3}{2}\right)  _{k}\left(  -2m\right)  _{k}%
}=\frac{m!\left(  n+1\right)  _{m}}{\left(  \frac{1}{2}\right)  _{m}\left(
n+\frac{3}{2}\right)  _{m}},\\
S\left(  m,n\right)  =\frac{\left(  m+1\right)  _{m}\left(  n+1\right)  _{n}%
}{\left(  m+n+2\right)  _{m}\left(  \frac{1}{2}\right)  _{n+1}}\frac{m!\left(
n+1\right)  _{m}}{\left(  \frac{1}{2}\right)  _{m}\left(  n+\frac{3}%
{2}\right)  _{m}}\\
=\frac{\left(  2m\right)  !\left(  n+1\right)  _{n}\left(  n+1\right)  _{m}%
}{\left(  m+n+2\right)  _{m}\left(  \frac{1}{2}\right)  _{m+n+1}\left(
\frac{1}{2}\right)  _{m}}=2^{2m+2n}\frac{m!\left(  n+1\right)  _{m}}{\left(
n+m+2\right)  _{m}\left(  n+\frac{1}{2}\right)  _{m+1}},
\end{gather*}
using $\left(  2m\right)  !=2^{2m}m!\left(  \frac{1}{2}\right)  _{m}$ and
$n!\left(  n+1\right)  _{n}=\left(  2n\right)  !=2^{2n}n!\left(  \frac{1}%
{2}\right)  _{n}.$ The last expression is equivalent to the stated formula in
the Proposition (typical step: $\left(  n+1\right)  _{m}=\frac{\left(
n+m\right)  !}{n!}$).
\end{proof}

\end{document}